\documentclass[11pt,reqno]{amsart}
\usepackage{cite}
\usepackage{geometry,graphicx, microtype}
\usepackage{amsmath,amssymb,subfig} 
\usepackage{bbm}
\usepackage{color}
\usepackage[table]{xcolor}
\usepackage{mathtools}
\usepackage[shortlabels]{enumitem}
\usepackage{hyperref}

\theoremstyle{plain}
\newtheorem{theorem}{Theorem}
\newtheorem{proposition}[theorem]{Proposition}
\newtheorem{lemma}[theorem]{Lemma}
\newtheorem{corollary}[theorem]{Corollary}

\theoremstyle{definition}
\newtheorem{definition}[theorem]{Definition}
\newtheorem{remark}[theorem]{Remark}

\newtheorem{example}[theorem]{Example}

\begin{document}

\title{Topological Mixing of Random Substitutions}
\author[E. D.~Miro, D.~Rust, L.~Sadun, G.~Tadeo]{Eden Delight Miro, Dan Rust, Lorenzo Sadun, Gwendolyn Tadeo}
\address{Department of Mathematics, Ateneo de Manila University, Quezon City 1108, Philippines}
\email{eprovido@ateneo.edu}
\address{School of Mathematics and Statistics, The Open University, Walton Hall, Milton Keynes, MK7 6AA, UK}
\email{dan.rust@open.ac.uk}
\address{Department of Mathematics, University of Texas at Austin, Austin TX 78712, USA}
\email{sadun@math.utexas.edu}
\address{Department of Mathematics, Saint Louis University, Baguio City, Philippines}
\email{gstadeo@slu.edu.ph}
\date{\today}

\begin{abstract}
We investigate topological mixing of compatible random substitutions. 
For primitive random substitutions on two letters whose second eigenvalue is greater than one in modulus, we identify a simple, computable criterion which is equivalent to topological mixing of the associated subshift.
This generalises previous results on deterministic substitutions.
In the case of recognisable, irreducible Pisot random substitutions, we show that the associated subshift is not topologically mixing.
Without recognisability, we rely on more specialised methods for excluding mixing and we apply these methods to show that the random Fibonacci substitution subshift is not topologically mixing.
\end{abstract}



\maketitle

\section{Introduction}
The question of what it means for a dynamical system to be \emph{disordered} is a subtle one with a rich history dating back to the birth of dynamics itself with Poincar\'{e}'s study of chaotic orbits in celestial mechanics \cite{P:poincare}.
Mathematicians measure disorder using a variety of tools including entropy, the dynamical spectrum, Lyapunov exponents, recurrence phenomena, and mixing properties.
While a system deemed to be disordered is ideally identified as so by each of these measures, this is not always the case.
Indeed, systems exist which are disordered according to one measure, but well-ordered according to another.
Subshifts arising from \emph{random substitutions} are prototypical examples.

First studied in the context of mathematical quasicrystals by Godr\`{e}che and Luck \cite{Godreche}, the random Fibonacci substitution has the curious property of giving rise to a dynamical system which is locally disordered but globally well-ordered, characterised by the simultaneous presence of positive topological entropy and a non-trivial pure-point component in its associated diffraction spectrum.
Their results immediately encourage the further study of the dynamics of the random Fibonacci substitution and related systems.
In this work, we consider one such dynamical property, namely topological mixing of the associated subshift and the associated tiling space.

Random substitutions are maps which send letters from a finite alphabet to finite collections of words over the same alphabet.
They have received renewed attention in recent years, following the work of Rust and Spindeler \cite{dynamical}, in which they initiated the study of random substitutions in the context of topological dynamics and ergodic theory and they posed numerous open questions and routes for further investigation.
Since then, random substitution subshifts have been studied in terms of (among others) their frequency measures \cite{ergodic}, diffraction spectrum \cite{BSS:rand-diffraction, Gohlke}, periodicity \cite{Rust}, automorphism groups \cite{FR2019}, topological and measure theoretic entropy \cite{G:entropy,GMRS:metric-entropy}, and relations to shifts of finite type \cite{sft}.
In their work, Rust and Spindeler determined by ad hoc methods that the subshift associated with the random period doubling substitution was not topologically mixing and they asked if a more general method could be established to determine when an arbitrary random substitution gives rise to a topologically mixing subshift.
We address that question here and give an answer for large families.

For deterministic substitutions, the classification of which substitution subshifts are topologically mixing is still incomplete.
Dekking and Keane studied topological mixing for deterministic substitutions in the 70s \cite{DK:mixing} and provided the first example of a minimal system which is topologically strongly mixing of order two but not order three.
Importantly, they also established that the presence of mixing depends not just on the abelianisation of the substitution, but also on the order of letters.
The most recent attempt at providing a full classification is the work of Kenyon, Sadun and Solomyak \cite{Kenyon}, where they were able to determine the presence of mixing for large families of substitutions (their conditions do depend only on the abelianisation).
While we do not address the same question (that is, in the deterministic setting), we use their results as a guide in our investigation of topological mixing of random substitutions, and as a tool in the study itself, capitalising on the close relationship between random substitutions and their deterministic counterparts.

In Section \ref{SEC:definitions}, we introduce the basic definitions of primitive and compatible random substitutions, as well as topological mixing in the context of subshifts over finite alphabets.
We also outline our main results and compare them with the corresponding results on deterministic substitutions established in \cite{Kenyon}.
In Section \ref{SEC:main-results}, we prove our main results on determining topological mixing for primitive compatible random substitutions when the second largest eigenvalue of the substitution matrix is greater than 1 in modulus.
In Section \ref{SEC:pisot}, we  show that if a primitive compatible random substitution is irreducible Pisot, then its subshift is $C$-balanced and use this to prove that irreducible Pisot random substitutions are not topologically mixing as long as a mild recognisability condition is satisfied.
In Section \ref{SEC:examples}, we exhibit examples for each of our main theorems to illustrate their application.
Significantly, we establish that the random Fibonacci substitution subshift is not topologically mixing; we weakened the recognisability condition enough in Section \ref{SEC:pisot} for the main result to apply to this case (the random Fibonacci substitution does not satisfy full recognisability).
Finally, in Section \ref{SEC:R-actions} we prove analogous theorems for 
$\mathbb{R}$-actions on random substitution tiling spaces. 

\section{Definitions and main results}\label{SEC:definitions}

\subsection{Random substitution subshifts}
Let $\mathcal{A}=\{a_1, a_2, \ldots, a_d\}$ be a finite alphabet.
For a word $w = w_1 \cdots w_n$ over $\mathcal{A}$, let $|w| \coloneqq n$ denote the length of $w$ and let $\mathcal{A}^n$ denote the set of all words whose length is equal to $n$.
Let $\mathcal{A}^+=\bigcup_{n=1}^\infty \mathcal{A}^n$ be the set of all non-empty words over $\mathcal{A}$. 
Appending the empty word $\epsilon$ then yields the free monoid $\mathcal{A}^*=\mathcal{A}^+\cup\{\epsilon\}$ of all finite words over $\mathcal{A}$ under concatenation. 
A \textit{subword} of a word $w=w_1\cdots w_n$ is a word $w_{[i,j]}=w_i\cdots w_j$ for some $1\le i\le j\le n$. 
The set $\mathcal{A}^\mathbb{Z}$ of all bi-infinite sequences over $\mathcal{A}$ forms a compact metrisable space under the product topology, where $\mathcal{A}$ is endowed with the discrete topology. 
The \textit{shift map} $\sigma \colon \mathcal{A}^\mathbb{Z} \rightarrow \mathcal{A}^\mathbb{Z}$ given by $\sigma(x)_n=x_{n+1}$ is then a homeomorphism.
A \textit{subshift} $X\subseteq \mathcal{A}^\mathbb{Z}$ is a closed non-empty subspace of $\mathcal{A}^\mathbb{Z}$ that is invariant under the shift action $\sigma$. That is, $\sigma(X)=X$. The \emph{language} of a subshift $\mathcal{L}(X)$ is the set of all subwords of elements of $X$,
\[
\mathcal{L}(X) = \{w \in \mathcal{A}^\ast \mid w \hbox{ is a finite subword of some }  x \in X\}.
\]

A \textit{deterministic substitution} $\theta$ on $\mathcal{A}$ is a map $\theta \colon 
\mathcal{A}\to\mathcal{A}^+$ and uniquely extends to a map $\theta \colon \mathcal{A}^+\to 
\mathcal{A}^+$ by concatenation $\theta(w_1\cdots w_n) \coloneqq\theta(w_1)\cdots \theta(w_n)$.  
The action of a deterministic substitution on a bi-infinite element $x\in \mathcal{A}^\mathbb{Z}$ is defined analogously, where the only ambiguity is where to place the origin. The usual convention is
for $\theta(x)_0$ to be the first letter of $\theta(x_0)$. 
Deterministic substitutions are well-studied \cite{Baake, Fogg}.
We are interested in a generalisation which assigns a finite set of words, rather than only a single word, to each letter of the alphabet.

\begin{definition}
Let $\mathcal{A}$ be a finite alphabet, and let $\mathcal{P}(\mathcal{A}^+)$ denote the power set of $\mathcal{A}^+$. 
A \textit{random substitution} is a map $\vartheta \colon \mathcal{A}\to\mathcal{P}(\mathcal{A}^+)\backslash \varnothing$ such that, for each $a \in \mathcal{A}$, $\vartheta(a)$ is a non-empty finite set.
We can extend $\vartheta$ to a function $\mathcal{A}^+ \to\mathcal{P}(\mathcal{A}^+)\backslash \varnothing$ by concatenation $$\vartheta(v_1\cdots v_n)=\vartheta(v_1)\cdots\vartheta(v_n) \coloneqq \left\{w_1 \cdots w_n \mid w_i \in \vartheta(v_i), \; 1\leq i \leq n \right\}.$$
We then extend $\vartheta$ to a function $\vartheta \colon \mathcal{P}(\mathcal{A}^+)\backslash \varnothing \to\mathcal{P}(\mathcal{A}^+)\backslash \varnothing$ by $\vartheta(B)\coloneqq \bigcup_{w\in B}\vartheta(w)$.
Consequently, we can now take powers of $\vartheta$ by composition and so 
$\vartheta^k \colon \mathcal{P}(\mathcal{A}^+)\backslash \varnothing \to\mathcal{P}(\mathcal{A}^+)\backslash \varnothing$ is defined for any $k \geq 0$, where $\vartheta^0\coloneqq \operatorname{Id}_{\mathcal{P}(\mathcal{A}^+)\backslash \varnothing}$ is the identity and $\vartheta^{k+1} \coloneqq \vartheta \circ \vartheta^k$. An element of $\vartheta^k(a)$,
where $a \in \mathcal{A}$, is called a \textit{super-word} of degree $k$. 
\end{definition}

\begin{definition}
A word $w\in \mathcal{A}^+$ is called a \textit{realisation of $\vartheta$} on a word $v\in \mathcal{A}^+$ if $w\in \vartheta(v)$. 
A word $w$ that is a realisation of $\vartheta$ on a legal word $v\in \mathcal{L}_\vartheta$ is called an \textit{inflation word}.
Similarly, a bi-infinite sequence $x\in \mathcal{A}^\mathbb{Z}$ is a \textit{realisation} of $\vartheta$ on a bi-infinite sequence $y\in \mathcal{A}^\mathbb{Z}$ if $x\in \vartheta(y)$.

We say that $\vartheta$ has \textit{constant length $\ell$} if all of the super-words of degree $1$ have the 
same length $\ell \ge 1$.
A deterministic substitution $\theta$ is called a \textit{marginal} of a random substitution $\vartheta$ if $\theta(a)\in \vartheta(a)$ for all $a\in\mathcal{A}$. 
The random substitution $\vartheta$ is then said to be a \textit{local mixture} of its marginals $\left\{\theta_i\right\}_{i\in I}$.
We say that a word $w \in \mathcal{A}^+$ is $\vartheta$\textit{-legal} if there is a natural number $k$ such that $w$ is a subword of some super-word of degree $k$.
\end{definition}

\begin{example}
The \emph{random Fibonacci substitution} is given by $\vartheta\colon a\mapsto \{ab, ba\}, b\mapsto \{a\}$ and was first introduced by Godr\'{e}che and Luck \cite{Godreche}. 
Note that $aa$ is a subword of $baa \in \vartheta(ab) \subset \vartheta^2(a)$, 
while $bb$ is a subword of $abba \in \vartheta(aa)$, and so is a subword of an element of 
$\vartheta^3(a)$. 
This makes both $aa$ and $bb$ $\vartheta$-legal. 
However, we will see that $bbb$ and $aaaaa$ are not $\vartheta$-legal. 

The marginals of the random Fibonacci substitution $\vartheta$ are given by the deterministic Fibonacci substitution
$\theta_1\colon a \mapsto ab, b\mapsto a$ and  its reflection $\theta_2\colon a \mapsto ba, b\mapsto a$. While the word $bb$ is $\vartheta$-legal, it is neither $\theta_1$-legal nor $\theta_2$-legal, since
the only way for $bb$ to be part of an inflation word is for the first $b$ to be part of $\theta_1(a)=ab$ and for the second $b$ to be part of $\theta_2(a)=ba$.

\end{example}

Let $\Phi \colon \mathcal{A}^*\to \mathbb{N}_0^d \colon w \mapsto \left(|w|_{a_1},|w|_{a_2},\dots,|w|_{a_d}\right)^T$ denote the \textit{abelianisation function}, where $|w|_a$ denotes the number of occurrences of $a$ in $w$. 
That is, $\Phi$ takes a word $w \in \mathcal{A}^*$ and enumerates the number of occurrences of each letter in $w$.

\begin{definition}
Let $\vartheta$ be a random substitution on $\mathcal{A}=\{a_1, a_2, \ldots, a_d\}$ and let $\{\theta_i\}_{i\in I}$ denote the set of marginals of $\vartheta$.
We say that $\vartheta$ is \textit{compatible} if the abelianisations of its marginals coincide.
That is, for each $a \in \mathcal{A}$, all words in $\vartheta(a)$ have the same abelianisation.  In that case, the \textit{substitution matrix $M_\vartheta$} of $\vartheta$ is given by $\left(M_\vartheta\right)_{ij} \coloneqq \left|\vartheta(a_j)\right|_{a_i}$ for all $1\leq i,j\leq d$.
That is, the common abelianisation of $\vartheta(a_i)$ determines the $i$th column of $M$ (not the
$i$th row). 

A matrix $M$ is {\em primitive} if there exists a power $p$ such that all entries of $M^p$ are positive.
We say that the compatible random substitution $\vartheta$ is primitive if $M_\vartheta$ is primitive.
In that case, each super-word of degree $p$ contains at least one copy of each letter. 
If $\vartheta$ is primitive, we let $\lambda_1 \coloneqq \lambda_{PF}$ denote the Perron--Frobenius (PF) eigenvalue of $M_\vartheta$, the unique largest eigenvalue. 
A compatible random substitution is \textit{Pisot} if $\lambda_1$ is a Pisot number---an algebraic integer greater than 1, all of whose algebraic conjugates lie in the open unit disk.
If the characteristic polynomial of $M_\vartheta$ is irreducible, then we likewise call $\vartheta$ irreducible.
\end{definition}
When a compatible primitive random substitution $\vartheta$ is defined on a two-letter alphabet,  
we refer to the PF-eigenvalue $\lambda_1$ as the \textit{first eigenvalue} of $M_\vartheta$ and to the other eigenvalue $\lambda_2$ as the \textit{second eigenvalue}.

\begin{example}
The random Fibonacci substitution $\vartheta\colon a\mapsto \{ab, ba\}, b\mapsto \{a\}$ is compatible as $\Phi(ab)=\Phi(ba)=(1,1)^T$ and $\Phi(b)= (1,0)^T$.
The corresponding substitution matrix is 
$\begin{psmallmatrix}
1 & 1 \\
1 & 0 \\
\end{psmallmatrix}$
with irreducible characteristic polynomial $\lambda^2-\lambda-1$, first eigenvalue $\lambda_1=\frac{1+\sqrt{5}}{2}$ (the golden ratio), and second eigenvalue 
$\lambda_2=\frac{1-\sqrt{5}}{2}$. Since $|\lambda_2|<1$, $\lambda_1$ is a Pisot number. Thus the random Fibonacci substitution is irreducible Pisot.  
\end{example}

\begin{definition}
The \textit{language of $\mathcal{L}_\vartheta$} of a random substitution  $\vartheta$ on $\mathcal{A}$ is
$$\mathcal{L}_{\vartheta}=\left\{ w \in \mathcal{A}^* \mid w \textrm{ is }\vartheta \textrm{-legal} \right\}.$$
The set of length-$n$ legal words for $\vartheta$ is denoted $\mathcal{L}_\vartheta^n \coloneqq \mathcal{L}_\vartheta \cap \mathcal{A}^n.$
The \textit{random substitution subshift} of $\vartheta$ (RS-subshift)
is 
\[
X_\vartheta \coloneqq \left\{x\in \mathcal{A}^\mathbb{Z} \mid  w \text{ is a finite subword of } x \Rightarrow w \in \mathcal{L}_\vartheta\right\}.
\]
\end{definition}
It is easy to verify that $X_\vartheta$ is a closed, shift-invariant subspace of the full shift $\mathcal{A}^\mathbb{Z}$, so $X_\vartheta$ is a subshift. A primitive compatible random substitution
gives rise to a non-empty RS-subshift, since each of its marginals already gives rise to a non-empty
subshift. 

A subspace $Y\subset X_\vartheta$ is called \textit{substitutive} if there exists a primitive deterministic substitution $\theta$ such that $Y=X_\theta$. 
Many results relating to the dynamics and topology of primitive RS-subshifts were presented in the works of Gohlke, Rust and Spindeler \cite{sft, dynamical} including the following results that will be useful later.

\begin{theorem}[\!\!\cite{dynamical, sft}]\label{densesubs}
	Let $\vartheta$ be  a primitive random substitution with a non-empty RS-subshift $X_\vartheta$. Then:
\begin{enumerate}[(a),leftmargin=25pt]
\item $X_\vartheta$ contains an element with a dense shift-orbit;
\item either $X_\vartheta$ is substitutive or there are infinitely many distinct substitutive subspaces of $X_\vartheta$;
\item the union of the substitutive subspaces of $X_\vartheta$ is dense in $X_\vartheta$.
\end{enumerate}
\end{theorem}

\subsection{Mixing of substitution subshifts}
In the remainder of this section, we outline the main results of our work.
First, we introduce the property in which we are principally interested.

\begin{definition}
A dynamical system $(X,T)$ is said to be \textit{topologically mixing} if for any two non-empty open subsets $U,V \subset X$, there exists a natural number $N$ such that for all natural numbers $n\geq N$, $T^n(U) \cap V \neq \varnothing$. 
\end{definition}

If $v=v_1\cdots v_{m}$ is a word of length $m$ and if $n$ is an integer, the {\em cylinder set}
$[v]_n$ is the set of $x \in X_\vartheta$ such that $x_{i+n-1}=v_i$ for $i=1,\ldots,m$. That is, it is 
the set of bi-infinite words that contain the word $v$ starting at position $n$. These cylinder sets
form a basis for the topology of a subshift $X$. If $\sigma \colon X \to X$ is the 
shift map, then for large $n$, the statement that $T^n([u]_{n_1}) \cap [v]_{n_2} \neq \varnothing$ is equivalent
to the existence of a word $uwv \in \mathcal{L}(X)$ where $uw$ has length $n+n_2-n_1$.   
This allows us to recast topological mixing for subshifts in purely combinatorial terms.

\begin{definition}
A subshift $(X, \sigma)$ is said to be \textit{topologically mixing} if for any two words $u,v \in \mathcal{L}(X)$, there exists a natural number $N$ such that for all natural numbers $n \geq N$, there exists a word $w$ of length $n$ such that $uwv \in \mathcal{L}(X)$. 
\end{definition}

Topological mixing is preserved under taking factor maps.
That is, if $f\colon X \rightarrow Y$ is a factor map of dynamical systems and $X$ is mixing, then $Y$ is also mixing.
In particular, topological mixing is preserved under topological conjugacy. 
It is well-known \cite{Marcus} that a shift of finite type $X_A$ is topologically mixing if and only if $A$ is a primitive matrix, where $A$ is the adjacency matrix of a directed graph $G$ and $X_A$ is the vertex-shift on $G$. 
It has been shown by Gohlke, Rust and Spindeler \cite{sft} that every topologically transitive shift of finite type can be realised up to topological conjugacy as a primitive (though often non-compatible) RS-subshift. 
This then gives us a class of RS-subshifts that are topologically mixing, including the full shift $\mathcal{A}^\mathbb{Z}$ and the golden mean shift $X_A$ with $A = \left(\begin{smallmatrix} 1 & 1\\ 1 & 0 \end{smallmatrix}\right)$.
There also exist RS-subshifts that are not topologically mixing such as the RS-subshift associated with the random period doubling substitution $\vartheta_{PD}\colon a \mapsto \{ab, ba\}, b \mapsto \{aa\}$ \cite{dynamical}. 
 
Recall that a subshift $X$ is aperiodic if $X$ contains no periodic orbits. We say that a random substitution $\vartheta$ is aperiodic if its associated RS-subshift $X_\vartheta$ is aperiodic (and similarly for deterministic substitutions).
Our main results generalise those appearing in the work of Kenyon, Sadun and Solomyak \cite{Kenyon} to the random setting.

\begin{theorem}[\!\!{\cite[Prop.\ 1.1]{Kenyon}}]\label{thm:KSS}
Let $\theta$ be a primitive aperiodic deterministic substitution on an alphabet $\mathcal{A}$.
If the subshift $X_\theta$ is topologically mixing, then  
$\gcd\left\{|\theta^n(a)|:a\in \mathcal{A}\right\}=1$ for all $n\ge 1$.
\end{theorem}

 When $|\lambda_2| > 1$ and the alphabet only consists of two letters, they were able to show that 
 this criterion is also sufficient for topological mixing.

\begin{theorem}[\!\!{\cite[Thm.\ 1.2]{Kenyon}}]\label{thm:KSS-iff}
Let $\theta$ be a primitive aperiodic deterministic substitution defined on a two-letter alphabet $\mathcal{A}$ with second eigenvalue $\lambda_2$ greater than 1 in modulus. 
The subshift $X_\theta$ is topologically mixing if and only if 
$\gcd\left\{|\theta^n(a)|:a\in \mathcal{A}\right\}=1$ for all $n\ge 1$.
\end{theorem} 

Our first goal is to extend Theorem~\ref{thm:KSS} to the random case. 
With this in mind, we introduce an additional condition in Section \ref{subsec:sufficient}  referred to as \textit{local recognisability}, which means that legal words of the RS-subshift can be uniquely \emph{desubstituted} given a sufficiently large legal neighbourhood of the word.
This condition is an analogue of local recognisability for deterministic substitutions which is equivalent to aperiodicity in the primitive deterministic setting, thanks to a celebrated result of Moss\'{e} \cite{Mosse}.
Unfortunately, no such result is known for random substitutions, and the situation must necessarily be more complicated as, in contrast to the deterministic case, recognisability does not follow from aperiodicity.
A classic example is given by the random Fibonacci substitution whose RS-subshift is aperiodic because its PF-eigenvalue $\lambda_1=\frac{1+\sqrt{5}}{2}$ is not an integer, but it is not recognisable \cite{Rust}. It is known that the set of periodic points of a primitive RS-subshift is either empty (hence, aperiodic) or dense \cite{dynamical}. 

By assuming local recognisability, which holds automatically for aperiodic deterministic substitutions, we are able to generalise Theorem~\ref{thm:KSS} to the random setting.

\begin{theorem}\label{mainthm-non-pisot-mixingtogcd}
Let $\vartheta$ be a primitive compatible random substitution on an alphabet $\mathcal{A}$ that is locally recognisable.
If the RS-subshift $X_\vartheta$ is topologically mixing, then $\gcd\left\{|\vartheta^n(a)|:a\in \mathcal{A}\right\}=1$ for all $n\ge 1$.
\end{theorem}

Our second main result is an extension of Theorem~\ref{thm:KSS-iff} to the random case.  
To prove the necessary direction,
we consider the substitutive subspaces of a given RS-subshift associated with the marginals of its powers.
In order to apply Theorem~\ref{thm:KSS-iff} directly, we need to show that each of these substitutive subspaces is aperiodic.
We do this by providing a total classification of periodicity for primitive substitutions on two letters in Proposition \ref{pro:periodicforms}. (To the best of our
knowledge, this classification is new.)

\begin{theorem}\label{mainthm-non-pisot-iff}
Let $\vartheta$ be a primitive compatible random substitution on a two-letter alphabet $\mathcal{A}$ that is locally recognisable with second eigenvalue $\lambda_2$ greater than 1 in modulus. 
The RS-subshift $X_\vartheta$ is topologically mixing if and only if $\gcd\left\{|\vartheta^n(a)| : a\in \mathcal{A}\right\}=1$ for all $n\ge 1$.
\end{theorem}

For a primitive compatible random substitution $\vartheta$ of constant length $\ell>1$, by compatibility we have $\gcd\left\{|\vartheta^n(a)|:a\in \mathcal{A}\right\}= \ell^n$ for all $n\ge 1$. 
The following then follows from Theorem~\ref{mainthm-non-pisot-mixingtogcd}.

\begin{corollary}\label{cor:conslen}
Let $\vartheta$ be a primitive compatible constant-length random substitution on an alphabet $\mathcal{A}$ that is locally recognisable. 
The RS-subshift $X_\vartheta$ is not topologically mixing. 
\end{corollary}

Note however that even if the PF-eigenvalue of a random substitution is an integer, the
substitution may not be constant length (see Example \ref{ex:integer-eigenvalue}).
In the case that $|\lambda_2| > 1$, we can extend the above corollary to the case of integral PF-eigenvalue using the following short argument.

\begin{proposition}\label{subsnot}
Let $\vartheta$ be a primitive compatible random substitution on a two-letter alphabet that is locally recognisable with integer first eigenvalue $\lambda_1$ and second eigenvalue $\lambda_2$ greater than 1 in modulus.  
Then the RS-subshift $X_\vartheta$ is not topologically mixing.
\end{proposition}
\begin{proof}
Let $\theta$ be a marginal of $\vartheta$. 
By compatibility, $\lambda_1$ is also the first eigenvalue of $\theta$.
By a classical argument \cite[Theorem 1]{Dekking}, there exists a primitive deterministic substitution $\theta_{cl}$ of constant length $\ell = \lambda_1$ such that $X_\theta$ and $X_{\theta_{cl}}$ are topologically conjugate.
Since $\vartheta$ is locally recognisable, the marginal $\theta$ is also locally recognisable, and so is aperiodic by Moss\'{e}'s theorem \cite{Mosse}.
Theorem~\ref{thm:KSS} then states that $X_{\theta_{cl}}$ is not topologically mixing, insofar as 
$\theta_{cl}$ is constant-length.
Thus, the subshift $X_\theta$ is also not topologically mixing.
By Theorem~\ref{thm:KSS-iff}, there must exist some $n \geq 1$ such that $\gcd\{|\theta^n(a)| : a \in \mathcal{A}\} > 1$. But then 
$\gcd\{|\vartheta^n(a)| : a \in \mathcal{A}\} = \gcd\{|\theta^n(a)| : a \in \mathcal{A}\} > 1$, so by Theorem \ref{mainthm-non-pisot-iff}, $X_\vartheta$ is not topologically mixing.
\end{proof}

Our next main result concerns mixing properties of Pisot random substitutions.  
First, we show in Proposition \ref{thm:pisot-balanced} that the irreducible Pisot property implies $C$-balancedness. That is, the number of occurrences of a letter in a legal word does not differ from the statistically expected proportion by more than a uniformly bounded value. 

Some well-known Pisot random substitutions such as the random Fibonacci substitution fail to be locally recognisable.
However, they often satisfy a weaker condition which we call ``admitting recognisable words at all levels''. 
For a Pisot random substitution $\vartheta$, this is enough to imply that $X_\vartheta$ is not 
topologically mixing via the following theorem:
\begin{theorem}\label{pisot-recogwords-nonmixing}
Let $\vartheta$ be a $C$-balanced random substitution.
There exists a constant $N$ such that if $\vartheta$ admits a level-$n$ recognisable word for some $n \geq N$, then the RS-subshift $X_\vartheta$ is not topologically mixing. 
\end{theorem}
In particular, we are able to show that the random Fibonacci substitution $\vartheta$ admits recognisable words at all levels (Example \ref{eg:randfib-RW}), so $X_{\vartheta}$ is not topologically mixing.
It should be noted that in the deterministic setting, Theorem \ref{pisot-recogwords-nonmixing} is automatic in the Pisot case thanks to the fact that, for minimal system, topological mixing implies topological weak mixing.
This is not necessarily the case for non-minimal systems such as RS-subshifts, and so we require a direct proof of non-mixing.

Finally, we turn our attention to tilings associated with our subshifts. 
To each letter $a_i$ we associate a labelled interval (tile) 
of length $\ell(a_i)=\ell_i$ and to each
(possibly infinite) word we associate a concatenation of the tiles corresponding to the letters. Instead of considering a 
$\mathbb{Z}$-action on a space of bi-infinite words, we consider an $\mathbb{R}$-action on a space
of tilings. Remarkably, the theorems for topological mixing of 
random substitution tilings are nearly identical to those of random substitution subshifts, 
and so are not repeated here, only
with the condition that $\gcd\left\{|\vartheta^n(a)| : a\in \mathcal{A}\right\}=1$ for all $n\ge 1$
replaced by the existence of two (or more) tiles whose lengths have an irrational ratio. 

\section{Topological mixing for general primitive random substitutions}\label{SEC:main-results}

Let $\vartheta$ be a primitive random substitution on a finite alphabet $\mathcal{A}$ and let $u$ be a $\vartheta$-legal word. 
By the definition of $\vartheta$-legality, there is a natural number $k$ such that  
$u$ is a subword of a particular realisation $w_{a} \in \vartheta^{k}(a)$ for some letter $a\in \mathcal{A}$. For each $b \ne a$, pick such a word $w_b \in \vartheta^k(b)$ arbitrarily. 
Now, define the deterministic substitution $\theta_u \colon \mathcal{A}\to\mathcal{A}^+$ by $\theta_u(a) = w_a$ and $\theta_u(b)=w_b$ for each $b\in \mathcal{A}\backslash\{a\}$. 
Clearly, $\theta_u$ is a primitive deterministic substitution containing $u$ as a $\theta_u$-legal word whose associated subshift $X_{\theta_u}$ is contained in $X_\vartheta$. Furthermore, the substitution
matrix of $\theta_u$ is $M_\vartheta^k$, which is primitive. 

This construction is taken from \cite{sft} where it was used to prove Theorem~\ref{densesubs}(c).
If a substitution $\theta$ is a marginal of a power of $\vartheta$, we call the subspace $X_{\theta} \subset X_\vartheta$ a \textit{basic subspace} of $\vartheta$.
All basic subspaces are substitutive by construction and 
each $X_{\theta_u}$ is a basic subspace.
The union of the basic subspaces of $\vartheta$ is then dense in $X_\vartheta$ since, 
by the construction
of the previous paragraph, every $\vartheta$-legal word $u$ appears in the language of the basic subspace $X_{\theta_u}$.
The basic subspaces of $X_\vartheta$ will play an important role in allowing us to convert mixing properties of deterministic substitutions to random substitutions.

\subsection{A sufficient condition}\label{subsec:sufficient}

To prove Theorem \ref{mainthm-non-pisot-mixingtogcd}, we need to introduce the concept of local recognisability. Roughly speaking, this means that if $x$ and $y$ are bi-infinite words and 
$x \in \vartheta(y)$, then we can deduce the value of $y_0$ from the restriction of $x$ to a region
of fixed size around the origin. To make this precise, we must first introduce the notions of 
inflation word decompositions and induced inflation word decompositions. 
We then prove intermediate results concerning these constructions that will be used in the proof of Theorem \ref{mainthm-non-pisot-mixingtogcd}.

\begin{definition}[\!\!\cite{FR2019}]
Let $\vartheta$ be a random substitution and let $u\in \mathcal{L}_{\vartheta}$ be a legal word.
Let $n\geq 1$ be a natural number. 
For words $u_i \in \mathcal{A}^+$, the tuple $[u_1,\ldots,u_{\ell}]$ is called a \emph{$\vartheta^n$-cutting} of $u$ if $u_1\cdots u_{\ell}=u$ and there exists a $\vartheta$-legal word $v=v_1\cdots v_{\ell}$ such that 
\begin{itemize}
    \item For $i=2,\ldots, \ell-1$, $u_i$ is a super-word of level $n$ built from the letter $v_i$.
    That is, $u_i \in \vartheta^n(v_i)$. Note that $u_i$ is a word, while $v_i$ is a single letter. 
    \item $u_1$ is the suffix of a super-word of level $n$ built from $v_1$, and 
    \item $u_\ell$ is the prefix of a super-word of level $n$ built from $v_\ell$. 
\end{itemize}
That is, $u$ is contained in a realisation of $\vartheta^n(v)$, which is a concatenation of 
$n$-super-words, with each of the interior $u_i$'s being one of those $n$-super-words. 

We call $v$ a \textit{root} of the $\vartheta^n$-cutting and we call $\left([u_1,\ldots,u_{\ell}],v\right)$ the corresponding \emph{level-$n$ inflation word decomposition} of $u$. If 
$u$ actually is a realisation of $\vartheta^n(v)$, so $u_1\in \vartheta^n(v_1)$ and 
$u_\ell \in \vartheta^n(v_\ell)$, then we say that $u$ is an {\em exact} level-$n$ inflation word.
Finally, we let $D_{\vartheta^n}(u)$ denote the set of all level-$n$ decompositions of the 
$\vartheta$-legal word $u$.
\end{definition}

\begin{example}\label{eg:decomroot}
Let $\vartheta \colon a\mapsto \{ab, ba\}, b\mapsto \{a\}$ be the random Fibonacci substitution.
The word $aab$ has two possible $\vartheta$-cuttings $[a,ab]$ and $[a,a,b]$ with each having two distinct associated roots. The set of level-$1$ inflation word decompositions of $aab$ is $$D_\vartheta(aab)=\left\{\left([a,ab],ba\right), \left([a,ab],aa\right), \left([a,a,b],bba\right), \left([a,a,b],aba\right)\right\}.$$
In this example, $abba$ is an exact level-$1$ inflation word and $abbaaaabbaaaabba$ is an exact level-$1$, -$2$, -$3$, and -$4$ inflation word. However, $bb$ is not an exact (level-$1$) inflation word, since any concatenation of $1$-super-words containing $bb$ must also contain some $a$'s. 

Note that having a unique $\vartheta$-cutting does not lead to the uniqueness of roots.
For example, consider $\vartheta^2\colon a\mapsto \{aba, baa, aab\}, b\mapsto \{ab,ba\}$.
The word $bb$ has a unique $\vartheta^2$-cutting $[b,b]$, 
but this $\vartheta^2$-cutting can come from four distinct inflation word decompositions:
\[
D_{\vartheta^2}(bb)=\left\{([b,b],aa), ([b,b],ab), ([b,b],ba), ([b,b],bb)\right\}.
\]
Similarly, having a unique root does not mean that one has a unique $\vartheta$-cutting.
For example, under the random period doubling substitution $\varrho \colon a\mapsto \{ab, ba\}, b\mapsto \{aa\}$, the word $bab$ can only come from the legal word $aa$ but it has two possible $\vartheta$-cuttings; specifically, $D_\varrho(bab)=\{([ba,b],aa), ([b,ab],aa)\}.$
\end{example}

Let $u$ be a $\vartheta$-legal word and let $u_{[i,j]}$
be a subword. An inflation word decomposition of
$u$ restricts to an inflation word decomposition
of $u_{[i,j]}$, which we call an {\em induced inflation word decomposition}. The idea is simple, but the precise definition is somewhat
technical: 
\begin{definition}[\!\!\cite{FR2019}]
Let $\vartheta$ be a random substitution and let $d = ([u_1, \ldots, u_\ell], v)\in D_{\vartheta}(u)$ be an inflation word decomposition of $u$. 
For $1\leq i \leq j \leq |u|-1$, we write $d_{[i,j]}$ for the \emph{induced inflation word decomposition on the subword} $u_{[i,j]}$, defined by
$$d_{[i,j]}=\left([u_1, \ldots, u_\ell], v\right)_{[i,j]}=\left([\hat{u}_{k(i)},u_{k(i)+1},\ldots,u_{k(j)-1},\hat{u}_{k(j)}], v_{[k(i),k(j)]}\right),$$ 
where $1 \leq k(i) \leq k(j) \leq \ell$ are natural numbers such that
$$\left|u_1\cdots u_{k(i)-1}\right| < i \leq \left|u_1\cdots u_{k(i)}\right| \text{ and } \left|u_1 \cdots u_{k(j)}\right|\leq j < \left|u_1 \cdots u_{k(j)+1}\right|,$$
$\hat{u}_{k(i)}$ is a suffix of $u_{k(i)}$ and $\hat{u}_{k(j)}$ is a prefix of $u_{k(j)}$ such that
$$\hat{u}_{k(i)}u_{k(i)+1}\cdots u_{k(j)-1}\hat{u}_{k(j)}=u_{[i,j]}.$$

\end{definition}

\begin{example}\label{eg:induceddecom}
Let $\vartheta \colon a\mapsto \{ab, ba\}, b\mapsto \{a\}$ be the random Fibonacci substitution.
The legal word $u=ababa \in \mathcal{L}_\vartheta$ has exactly five level-$1$ inflation word decompositions given by 
$$D_\vartheta(u)=\left\{ ([a,ba,ba],baa), ([a,ba,ba],aaa), ([ab,ab,a],aab), ([ab,ab,a],aaa), ([ab,a,ba],aba)\right\}.$$
For the subword $u_{[2,4]}=bab$ of $u$, the first two elements of $D_\vartheta(u)$ yield the induced decomposition 
$d^{(1)}_{[2,4]}=([ba,b],aa)$, the next two elements yield $d^{(2)}_{[2,4]}=([b,ab],aa)$, while the last element yields $d^{(3)}_{[2,4]}=([b,a,b],aba)$. 
Thus, $\#\left\{d_{[2,4]} \mid d\in D_\vartheta(u)\right\}=3$. 
In this example, all three possible $\vartheta$-cuttings
of $bab$ are induced from cuttings of $u$. 

By contrast, the word $u'=bbaba$ has a unique inflation word decomposition $d' = \left( [b,ba,ba], aaa \right)$ which yields a unique induced inflation word decomposition on the subword $u'_{[2,4]} = bab$ given by $d'_{[2,4]} = \left([ba,b],aa \right)$. That is, when
$bab$ sits inside $u$, the embedding tells us nothing
about the $\vartheta$-cutting of $bab$, but when 
$bab$ sits inside $u'$, the embedding uniquely defines
the $\vartheta$-cutting of $bab$.

\end{example}

\begin{definition}[\!\!\cite{FR2019}]
Let $\vartheta$ be a random substitution and let $u\in \mathcal{L}_\vartheta$ be a legal word. 
We say that $u$ is \emph{recognisable} if there exists a natural number $N$ such that for each legal word of the form $w=u^{(l)}uu^{(r)}$ with $\left|u^{(l)}\right|=\left|u^{(r)}\right|=N$, all inflation word decompositions of $w$ induce the same inflation word decomposition of $u$. That is, knowing
the $N$ letters to the left of $u$ and the $N$ letters
to the right of $u$ determines a unique induced inflation
word decomposition of $u$. 
We call the minimum such $N$ the \emph{radius of recognisability} for $u$.
If $u$ is recognisable with respect to the $n$th power $\vartheta^n$ of the random substitution, then we say that $u$ is \emph{level-$n$ recognisable} with respect to $\vartheta$.
\end{definition}

\begin{example} \label{eg:RW}
Let $\vartheta\colon a\mapsto \{ab, ba\}, b\mapsto \{a\}$ be the random Fibonacci substitution with
marginals $\theta_1$ and $\theta_2$ such that $\theta_1(a)=ab$ and $\theta_2(a)=ba$.
The word $w=abbaaaabbaaaabba$ is a level-1, -2, -3, and -4 recognisable word with radius 0 at each level, as it has the following sets of unique inflation word decompositions
$$\begin{array}{lcl}
D_{\vartheta}(w) & = &\left\{([ab,ba,a,a,ab,ba,a,a,ab,ba], aabbaabbaa)\right\},\\
D_{\vartheta^2}(w) & = & \{([ab,baa,aab,baa,aab,ba], baaaab)\}, \\
D_{\vartheta^3}(w) & = & \{([abbaa,aab,baa,aabba], abba)\}, \\
D_{\vartheta^4}(w) & = & \{([abbaaaab,baaaabba], aa)\}. \\
\end{array}$$
By contrast, the word $bab$ turns out not to be 
recognisable at level 1 (or at any level, actually), 
as there exist arbitrarily long legal
extensions of $bab$ that admit both an inflation word decomposition that induces the decomposition $([ba,b],aa)$ of $bab$, and another decomposition that induces $([b,ab],aa)$. Specifically, if $n$ is an odd natural number, let $u=\theta_1^n(aa)$ with the first letter 
removed, which is the same as $\theta_2^n(aa)$
with the last letter removed. The middle three letters of 
$u$ are $bab$. The inflation word decomposition of $u$ 
that comes from applying $\theta_2$ to $\theta_2^{n-1}(aa)$
induces the decomposition $([ba,b],aa)$, while the 
decomposition of $u$ that comes from applying $\theta_1$
to $\theta_1^{n-1}(aa)$ induces $([b,ab],aa)$.

Although the random Fibonacci substitution has $\vartheta$-legal words like $bab$ that are not recognisable, for each natural number $n$ there also exist $\vartheta$-legal words that are recognisable at level $n$. This fact, which is sufficient to establish many 
mixing properties of $X_\vartheta$, will be proven in Section \ref{SEC:examples}.
\end{example}

\begin{definition}
We call a random substitution $\vartheta$ \emph{locally recognisable} if there exists a natural number $N$ such that every $\vartheta$-legal word is recognisable with radius 
at most $N$.  
The minimum such $N$ is called the \textit{radius of recognisability for $\vartheta$}.
Similarly, for any natural number $n\geq 1$, 
if $\vartheta^n$ is locally recognisable, we denote by $N_{\vartheta^n}$ its radius of recognisability.
\end{definition}

\begin{remark}
As in the deterministic setting, there is a notion of global recognisability for random substitutions, first appearing in the work of Rust \cite{Rust}.
A random substitution $\vartheta$ is globally recognisable if for each bi-infinite word $x \in X_\vartheta$ there is a unique preimage $y \in X_\vartheta$ and a unique integer $n$ in a fixed bounded range such that $\sigma^n(x) \in \vartheta(y)$.
Global and local recognisability are in fact equivalent, 
but in practice it is usually much easier to check local recognisability than global.
\end{remark}

By a routine inductive argument, or by invoking the equivalence with global recognisability, one can show that local recognisability for compatible random substitutions is preserved under taking powers.

\begin{proposition}\label{localglobal}
A compatible random substitution $\vartheta$ is locally recognisable  if and only if $\vartheta^n$ is locally recognisable for all $n\ge1$. \hfill \qed
\end{proposition}

Unfortunately, local recognisability of a random substitution is a rather strong condition; many examples,
such as the random Fibonacci substitution, are not locally
recognisable. However, for many proofs it is enough that
a substitution admits recognisable words at all levels. 
We will return to this concept when we address the proof of Theorem \ref{pisot-recogwords-nonmixing}.

We are now in a position to prove the first of our 
main results, Theorem \ref{mainthm-non-pisot-mixingtogcd}.
For a random substitution $\vartheta$, we write $|\vartheta|:=\max \left\{|u| \mid u \in \vartheta(\mathcal{A}) \right\}$.

\begin{proof}[Proof of Theorem \ref{mainthm-non-pisot-mixingtogcd}]
We prove the contrapositive. Suppose that $\gcd\big\{|\vartheta^n(a)| : a\in \mathcal{A}\big\}=p > 1$ for some natural number $n \geq 1$.
As $\vartheta$ is locally recognisable, $\vartheta^n$ is also locally recognisable by Proposition \ref{localglobal}. So, without loss of generality, we can assume that $n = 1$. 
Let $N$ be the radius of recognisability for $\vartheta$,
let $u$ be a $\vartheta$-legal word of length
 $|u| > 2N+2|\vartheta|$, 
 and let $u'$ be the subword of $u$
 obtained by deleting the first $N$ and last $N$ letters of $u$.
 By local recognisability, every inflation word decomposition of $u$ induces the same inflation word 
 decomposition of $u'$. Since $|u'|>2|\vartheta|$, this
 unique decomposition contains a complete $1$-super-word $u''$ in a fixed position relative to $u'$ (and hence relative to $u$).
 
 Now suppose that $uwu$ is a $\vartheta$-legal word and 
 consider an inflation word decomposition of $uwu$. 
 This induces inflation word decompositions of the first
 and second $u$, giving rise to appearances of $u''$
 spaced exactly $|u|+|w|$ apart. However, the length of 
 every $1$-super-word is a multiple of $p$, so $|u|+|w|$
 must be a multiple of $p$. Since $p>1$, there are arbitrarily 
 large values of $|w|$ that are impossible, so $X_\vartheta$
 is not topologically mixing. 
 \end{proof}

\subsection{Substitutions on two letters}
In order to apply Theorem \ref{thm:KSS}, it will be helpful to classify all two-letter primitive substitutions whose subshifts are periodic. A fully detailed proof appears in the thesis of the fourth author \cite{Tadeo2019} and so we only present a sketch of the proof here.
The constant length version of this result appears in work of Baake, Coons and Ma\~{n}ibo \cite{Baake2}.
Our proof for the general case is similar in spirit to theirs.
In particular, we capitalise on the absence of asymptotic pairs for periodic subshifts.

\begin{theorem}\label{pro:periodicforms}
Let $\theta$ be a primitive deterministic substitution on $\mathcal{A}=\{a,b\}$. $X_\theta$ is periodic if and only if $\theta$ takes one of the following forms $\{a\mapsto u^k, b\mapsto u^\ell\}$, $\{a\mapsto (ab)^ka, b\mapsto (ba)^\ell b\}$, or $\{a\mapsto (ba)^kb, b\mapsto (ab)^\ell a\}$
where $u$ is a non-empty finite word and $k, \ell \ge 0$.
\end{theorem}
\begin{proof}
It is an easy exercise to show that the three above forms lead to periodic subshifts. The non-trivial step is in showing that no other form of substitution on two-letters gives rise to a periodic subshift.

Let $\theta$ be a primitive periodic deterministic substitution. Suppose that $\theta$ is not of the form $a \mapsto (ab)^ka, b \mapsto (ba)^\ell b$ or $a \mapsto (ba)^kb, b \mapsto (ab)^la$. It follows then that either $aa$ or $bb$ is a $\theta$-legal word. Without loss of generality, assume that $aa$ is legal. By periodicity  and primitivity of $\theta$, we must then have $aa,ab,ba \in \mathcal{L}_\theta$.

We first put the substitution into a standard form such that the leftmost letters of $\theta(a)$ and $\theta(b)$ are different. To do this, we take `conjugates' of the substitution. That is, if $\theta(a) = xv$ and $\theta(b) = xw$ for $x \in \{a,b\}$, then we replace the substitution with the conjugate substitution $\theta'\colon a \mapsto vx, b \mapsto wx$. It is easy to see that conjugating a primitive substitution leaves the associated subshifts fixed, so $X_\theta = X_{\theta'}$. We may repeat the conjugation process. Either the iteration never stops, in which case $\theta$ is of the form $a \mapsto u^k, b \mapsto u^\ell$, or we eventually reach a conjugate $\theta^{(n)} \colon a \mapsto x v_a, b \mapsto y v_b$ with $x \neq y$. By squaring if necessary, we can assume that $x = a$ and $y = b$. This new substitution $\theta \colon a \mapsto a u_a s, b \mapsto b u_b t$ for $s, t \in \{a,b\}$ is then the standard form we may assume our substitution to be in.

Recall that $x, y \in X$ are called an \emph{asymptotic pair} if $x \neq y$ and either $d(\sigma^n(x),\sigma^n(y)) \to 0$ or $d(\sigma^{-n}(x),\sigma^{-n}(y)) \to 0$ as $n \to \infty$.
Periodic subshifts necessarily admit no asymptotic pairs.
We now consider the possible cases for the letters $s,t$. In each case, the existence of an asymptotic pair $x,y \in X_\theta$ will allow us to conclude that $X_\theta$ is aperiodic and hence reach a contradiction. For the case $s=a,t=b$, the existence of $a.a$ as a legal seed for producing an asymptotic pair of fixed points for the substitution is essential. This is why we had to conclude that $aa,ab,ba \in \mathcal{L}_\theta$ earlier. 

We will show this case. The other cases are similar (possibly with different seeds).

\textit{Case: $s = a$, $t = b$.} The substitution admits two (distinct) fixed points (because $a.a$ and $a.b$ are both legal seeds)
\[
\begin{array}{l}
 \cdots \theta^2(au_a) \: \theta(au_a) \: au_a \: a \: . \: a \: u_aa \: \theta(u_aa) \: \theta^2(u_aa) \cdots\\
 \cdots \theta^2(au_a) \: \theta(au_a) \: au_a \: a \: . \: b \:\, u_bb \:\, \theta(u_bb) \:\, \theta^2(u_bb)\, \cdots
\end{array}
\]
which are elements in $X_\theta$ that are right-asymptotic to one another.
\end{proof}

\begin{corollary}
Let $\theta$ be a primitive periodic deterministic substitution on $\mathcal{A} = \{a,b\}$.
The second eigenvalue $\lambda_{2}$ of the substitution matrix $M_\theta$ is either $0, 1,$ or $-1$.
\end{corollary}
\begin{proof}
By Theorem \ref{pro:periodicforms}, we only need to consider the three possible forms that $\theta$ can take.
If $\theta$ is of the form $a \mapsto u^k, b \mapsto u^\ell$, then $M_\theta$ has rank one, with $\lambda_2=0$.
For the other forms, the substitution matrix has 
$(1,-1)$ as a left-eigenvector with eigenvalue $\pm 1$.
\end{proof}

\begin{lemma}\label{aperiodsubs}
Let $\vartheta$ be a primitive compatible random substitution defined on $\mathcal{A}=\{a,b\}$ with $|\lambda_2| > 1$. 
If $\gcd\{|\vartheta^n(a)|, |\vartheta^n(b)|\}=1$ for all $n\ge 1$ then the basic subspaces of $\vartheta$ are topologically mixing.
\end{lemma}
\begin{proof}
Let $X_{\theta}$ be a basic subspace of $\vartheta$. 
So, $\theta(a) \in \vartheta^k(a)$ and $\theta(b) \in \vartheta^k(b)$ for some natural number $k\ge 1$.
The second eigenvalue of $\theta$ then 
has modulus $|\lambda_2^k|>1$. In particular,
this eigenvalue is not 0 or $\pm1$, so $X_\theta$
is non-periodic. By the compatibility of $\vartheta$, we have
\[
\gcd\big\{|\theta^n(a)|,|\theta^n(b)|\big\}=\gcd\big\{|\vartheta^{nk}(a)|,|\vartheta^{nk}(b)|\big\}=1
\]
for all $n\ge 1.$ 
By Theorem \ref{thm:KSS}, $X_{\theta}$ is then topologically mixing. 
\end{proof}

We now prove Theorem \ref{mainthm-non-pisot-iff}, the generalisation of Theorem \ref{thm:KSS-iff} to the random setting.

\begin{proof}[Proof of Theorem \ref{mainthm-non-pisot-iff}]
One direction is immediate from Theorem \ref{mainthm-non-pisot-mixingtogcd}. So, it remains to show that if $\gcd\{|\vartheta^n(a)|, |\vartheta^n(b)|\}=1$ for all $n \geq 1$, then $X_\vartheta$ is topologically mixing.

Let $u$ and  $v$ be arbitrary legal words. 
There then exist 
natural numbers $k_u$ and $k_v$ such that $u$ is a subword of a legal $k_u$-super-word and $v$ is a subword of a 
legal $k_v$-super-word. Pick $k$ bigger than the larger of 
$k_u$ and $k_v$. If $a_ia_j$ is a legal 2-letter word then,
by primitivity, there is an element of $\vartheta^k(a_ia_j)$ that contains $u$ in the first $k$-super-word and $v$ in the second. Thus there is a $\vartheta$-legal word of the form
$uwv$. 

Let $\theta\coloneqq \theta_{uwv}$ be a marginal of a power of $\vartheta$ such that $uwv$ is $\theta$-legal.
By Lemma \ref{aperiodsubs}, $X_\theta$ is topologically mixing.
Thus there exists a natural number $N > 0$ such that for all $n \geq N$, there is a word $w_n$ of length $n$ such that $uw_nv$ is $\theta$-legal, and hence $\vartheta$-legal.
As $u$ and $v$ were chosen arbitrarily, it follows that $X_\vartheta$ is topologically mixing.
\end{proof}

\section{Pisot random substitutions}\label{SEC:pisot}
\subsection{\texorpdfstring{$C$}{}-Balancedness}
The concept of `balancedness' was first introduced by Morse and Hedlund \cite{MorseHedlund} to provide one of the many classifications of Sturmian sequences \cite{Fogg}. Here, we consider the more general notion of \emph{$C$-balanced} sequences, also called sequences of \emph{bounded discrepancy}. For primitive substitutions, $C$-balancedness was classified by Adamczewski \cite{Adamczewski} and for the most part, the story changes little in the study of $C$-balancedness for compatible random substitutions. Note, however, that we do not attempt a full classification of $C$-balancedness in this setting. We only consider the Pisot case, which is sufficient for our needs. An approachable account of $C$-balancedness is presented in the work of Berth\'{e} and Bernales \cite{Berthe2}, from which we adopt the following definitions and results.
\begin{definition}
A bi-infinite sequence $x\in \mathcal{A}^\mathbb{Z}$ is said to be \textit{$C$-balanced} if there exists a constant $C$ such that for every $a \in \mathcal{A}$ and all pairs $w, w'$ of subwords of $x$ with $|w|=|w'|$,
\[
\left||w|_{a}-|w'|_{a}\right|\le C.
\]
A subshift $X\subseteq \mathcal{A}^\mathbb{Z}$ is said to be \textit{$C$-balanced} if there exists a constant $C$ such that for every $a \in \mathcal{A}$ and all pairs $w, w' \in \mathcal{L}(X)$ with $|w| = |w'|$,
\[
\left||w|_{a} - |w'|_{a}\right|\le C.
\] 
\end{definition}

\begin{remark}
Note that we only consider $C$-balancedness with respect to letters, and not the more general notion of $C$-balancedness with respect to subwords, which is a uniform bound on the deviation of the number of appearances of any legal subword within pairs of words of the same length.
Indeed, for our setting, this would be asking for too much from random substitutions, where such a property necessarily requires the corresponding subshift to be uniquely ergodic.
This is never the case for non-degenerate RS-subshifts.
\end{remark}

Recall that the \textit{frequency} of a letter $a\in \mathcal{A}$ in a bi-infinite sequence $x \in \mathcal{A}^\mathbb{Z}$ is the limit
$$f_a\coloneqq\lim_{n\to \infty}\frac{\left|x_{[-n,n]}\right|_a}{2n+1}, $$
if it exists, and $x$ is said to have \textit{uniform letter-frequencies} if, for every letter $a\in \mathcal{A}$, the convergence of $\frac{|x_{k}\cdots x_{k+2n}|_a}{2n+1}$ towards $f_a$ is uniform in $k$, when $n$ tends to infinity.
The amount $\big||w|_{a}-f_{a}|w| \big|$ that a subword $w$ of $x$ deviates from the expected number of occurrences of the letter $a$, according to the letter-frequency of $a$ in $x$ is called the \emph{discrepancy} of the letter $a$ in the word $w$.

\begin{theorem}[\!\!{\cite[Prop.\ 2.4]{Berthe3}}]\label{thmbalanced}
A bi-infinite sequence $x \in \mathcal{A}^\mathbb{Z}$ is $C$-balanced for some $C$ 
if and only if it has uniform letter frequencies and there exists another constant $B$ such that for any finite subword $w$ of $x$ and any $a \in \mathcal{A}$,  
\[
\big||w|_{a}-f_{a}|w| \big|\le B,
\] where $f_a$ is the frequency of $a$.
\end{theorem}

\begin{theorem}[\!\!{\cite[Thm.\ 2.5]{Berthe2}}]\label{balancedmain}
Let $\theta$ be a primitive deterministic irreducible Pisot substitution. Then the subshift $X_\theta$ is $C$-balanced for some $C$.
\end{theorem}

We now extend this result to the random setting.
The proof is a simple adaptation of the proof in the deterministic setting \cite{Berthe3}, generalising the usual Dumont--Thomas prefix-suffix decomposition \cite{DumontThomas} to the setting of random substitutions.
\begin{theorem}\label{thm:pisot-balanced}
Let $\vartheta$ be a compatible primitive irreducible Pisot random substitution. Then the RS-subshift $X_\vartheta$ is $C$-balanced for some $C$.
\end{theorem}
\begin{proof}
As $\vartheta$ is irreducible Pisot, we know that the second-largest eigenvalue $\lambda_2$ has modulus strictly bounded above by $1$ and that the substitution matrix 
is diagonalisable.
Let $|\vartheta|=k$ denote the longest inflation word for $\vartheta$.
By diagonalising $M_\vartheta$, a routine calculation shows that for any letters $a,b \in \mathcal{A}$, the discrepancy of $a$ in a level-$n$ inflation word in $\vartheta^n(b)$ is bounded by $D|\lambda_2|^n$ for some fixed constant $D$.

Now consider an arbitrary $\vartheta$-legal word $w$ and an inflation word decomposition that includes a complete 
$n$-super-word with $n$ as large as possible. 
That is, $w$ fully contains between $1$ and $2(k-1)$ complete $n$-super-words, followed by a (possibly empty) suffix $s_n$ and preceded by a (possibly empty) 
prefix $p_n$ such that $s_n$ and $p_n$ are both strictly contained in $n$-super-words. 
The suffix $s_n$ then consists of at most $k-1$ $(n-1)$-super-words followed by a further suffix 
$s_{n-1}$ that is strictly contained in an $(n-1)$-super-word. Continuing in this way, $s_{n-r}$ comprises at most $(k-1)$ $(n-r-1)$-super-words and 
a further suffix $s_{n-r-1}$. 
Iterating, we get that $s_n$ is the concatenation of 
up to $k-1$ $(n-1)$-super-words, up to $k-1$ 
$(n-2)$-super-words, etc., ending with up to $k-1$ 
individual letters. Likewise, the prefix $p_n$ is the 
concatenation of up to $k-1$ letters, up to $k-1$ $1$-super-words, up to $k-1$ $2$-super-words, continuing through $(n-1)$-super words. 

Since the discrepancy of $w$ is bounded by the sum of the 
discrepancies of its pieces, and since the discrepancy of 
each $i$-super-word is bounded by $D |\lambda_2|^i$, the total discrepancy $\big| |w|_a - f_a|w| \big|$ of $a$ in $w$ is bounded above by 
\[
\sum_{i=0}^{n} 2(k-1)D|\lambda_2|^i < \sum_{i=0}^\infty 
2(k-1)D |\lambda_2|^i 
=\frac{2(k-1)D}{1-|\lambda_2|}.
\]
\end{proof}

\begin{remark} This proof can be modified to show that any compatible random substitution with
$|\lambda_2|<1$ must be $C$-balanced. This class includes examples, like the random Thue-Morse
substitution, that have $0$ as an eigenvalue. The complication is that the 
substitution matrix need not be diagonalisable, in that the eigenvalue 0 
(and only the eigenvalue 0!) can be associated with a non-trivial Jordan block. 
To show $C$-balancedness, the case $\lambda_2=0$ has to be treated separately. 
\end{remark}

We remark that, as a consequence of Theorem \ref{thm:pisot-balanced}, bounded Rauzy fractals exist for irreducible Pisot random substitutions (and are almost surely independent of the choice of element chosen from the subshift).
This potentially offers a new tool in the effort to solve the long-standing Pisot conjecture \cite{ABBLS:pisot}.
We therefore invite further exploration of these random substitution Rauzy fractals.

\subsection{Non-mixing of Pisot random substitution subshifts}
We are now in a position to prove Theorem \ref{pisot-recogwords-nonmixing}.
Since irreducible Pisot substitutions are $C$-balanced by Theorem \ref{thm:pisot-balanced}, it is sufficient to prove the following.
\begin{theorem}\label{thm:No-mixing-balanced}
Let $\vartheta$ be a primitive compatible random substitution such that $X_\vartheta$ is $C$-balanced.
There exists a constant $N$ such that if $\vartheta$ admits a level-$n$ recognisable word for some $n \geq N$, then the RS-subshift $X_\vartheta$ is not topologically mixing.
\end{theorem}

\begin{proof}
This proof has several steps: 
\begin{enumerate}
\item Showing that the existence of a level-$n$ recognisable word implies the existence of a $\vartheta$-legal word $u$ such that, if $uwu$ is $\vartheta$-legal, then the length
of $uw$ is the length of an exact level-$n$ inflation word. This step does not require $C$-balancedness.
    \item Showing that for each length $L$ and natural number $n$, the set of 
    abelianisations of words of the form $\vartheta^n(v)$,
    where $v$ is a $\vartheta$-legal word of length $L$, 
    is bounded by a constant independent of $L$ and $n$. 
    This is where $C$-balancedness comes in.
    \item Showing that for sufficiently large $n$, the set of possible lengths of level-$n$ inflation words has natural density strictly less than 1. This step uses
    step 2. Combined with step 1, this shows that $X_\vartheta$ is not topologically mixing.
\end{enumerate}

Step 1: Suppose that there exists a level-$n$ recognisable 
word $\hat u$ with recognisability radius $R$. Let $u_l$ 
and $u_r$ be words of length $R$ such that $u=u_l\hat u u_r$ is 
$\vartheta$-legal. By recognisability, all level-$n$ 
inflation word decompositions of $u$ result in the same
decomposition of $\hat u$. 

Now consider a level-$n$ decomposition of a $\vartheta$-legal word $uwu$. Since the decompositions of 
the two copies of $\hat u$ are identical, there are 
corresponding points in the two copies that are spaced by 
an exact level-$n$ inflation word. However, the spacing 
between corresponding points in the two copies of $\hat u$ 
is precisely the length of $uw$. 

Step 2: Suppose that our alphabet $\mathcal{A}$ consists of $d$ letters. 
The number of $a$'s in a word of fixed length $L$
can take on at most $C+1$ values, since the largest such value is at most $C$ plus the smallest. 
Repeating for all but one letter and noting that the sum of the entries of the ablianisation is fixed, 
we see that the set of abelianisations of words of length exactly
$L$ has cardinality at most $(C+1)^{d-1}$. 
Since the abelianisation of each element of $\vartheta^n(v)$ is just 
$M_\vartheta^n$ times the abelianisation of $v$, there are
at most $(C+1)^{d-1}$ abelianisations of exact level-$n$ inflation words whose roots have length exactly $L$. 

Step 3: The growth in the lengths of $n$-super-words is 
controlled by $\lambda_1$, the PF-eigenvalue of $M_\vartheta$. Specifically, there exist positive constants $c_1$
and $c_2$ such that every $n$-super-word has length at 
least $c_1 \lambda_1^n$ and at most $c_2 \lambda_1^n$. 
If $w$ is a $\vartheta$-legal exact level-$n$ 
inflation word with root $v$ and if $|w| < L$, then 
$|v| < L \lambda_1^{-n}/c_1$. Since each possible length of 
$|v|$ gives rise to at most $(C+1)^{d-1}$ possible abelianisations of $w$, and since each abelianisation determines a unique length, the number of possible lengths 
of $w$ is bounded by $(C+1)^{d-1} L \lambda_1^{-n}/c_1$. As $L$ goes to $\infty$, the density of possible lengths of $w$ is bounded by $(C+1)^{d-1}/(c_1 \lambda_1^n)$. If we pick $N$ such
that $\lambda_1^N > (C+1)^{d-1}/c_1$, then for all $n \geq N$ this
density is less than 1. 
\end{proof}

\section{Additional Examples}\label{SEC:examples}
\begin{example}
Consider the compatible random substitution $\vartheta \colon a\mapsto \{abababa, bbbaaaa\}, b\mapsto \{babb,bbab\}$
whose substitution matrix is 
$\begin{psmallmatrix}
4&1\\
3&3\\
\end{psmallmatrix}$
with eigenvalues $\lambda_1=\frac{1}{2}(7+\sqrt{13})$ and $\lambda_2=\frac{1}{2}(7-\sqrt{13})>1$. 
It can be verified that $\gcd\{|\vartheta^n(a_i)| : a_i\in \mathcal{A}\} = 1$ for every $n\ge 1$. 
Thus, by Theorem \ref{mainthm-non-pisot-iff}, $X_\vartheta$ is topologically mixing. 
\end{example}

\begin{example}\label{ex:integer-eigenvalue}
Consider the random substitution $\vartheta \colon a\mapsto \{bb\}, b\mapsto \{abaaba, ababaa\}.$
Its substitution matrix is 
$\begin{psmallmatrix}
0&4\\
2&2\\
\end{psmallmatrix}$
with eigenvalues $\lambda_1=4$ and $|\lambda_2|=2 >1$. 
Here $\gcd\{|\vartheta(a_i)|:a_i\in \mathcal{A}\}=
\gcd\{2,6\}=2>1$. 
Using the fact that the word $bb$ can only come from a 
substituted $a$, one can easily show that $\vartheta$ is locally recognisable.
Thus, as a consequence of Theorem \ref{mainthm-non-pisot-mixingtogcd}, the subshift $X_\vartheta$ is not topologically mixing. This can also
be seen directly, since the spacing between any two 
appearances of $bb$ must be a 1-inflation word and so must
have even length. 
\end{example}

As previously noted, the random Fibonacci substitution is not locally recognisable. Likewise, neither are any of the random metallic means substitutions \cite{Zeckendorf}.
However, we do not need local recognisability to prove 
a lack of topological mixing. By 
Theorem \ref{pisot-recogwords-nonmixing},
we just need to show that 
there exist recognisable words at a sufficiently high level.
For the random Fibonacci substitution, we will do more,
showing that recognisable words exist at {\em all} levels. 

\begin{example}\label{eg:randfib-RW}
Let $\vartheta$ be the random Fibonacci substitution given by
$$\vartheta\colon a  \mapsto \{ab,ba\}, b \mapsto \{a\}.$$
Note that all four 2-letter words are $\vartheta$-legal.  
We begin with the seeds $F_0=a|a$ and $F_1=ab|ba$. We
successively construct
words $F_n =L_n|R_n \in \vartheta(F_{n-1})$ as follows, where the bar divides $F_n$ into level-$n$ super-words. 
Let $\phi$ be a map that
\begin{itemize}
    \item Replaces each $b$ with an $a$. 
    \item Replaces each $a$ that is followed by a $b$ with
    $ba$.
    \item Replaces each $a$ that is preceded by a $b$ with $ab$.
    \item Replaces each $a$ that is preceded by an $a$ with 
    $ab$ if the preceding $a$ was replaced with $ba$, and with $ba$ if the preceding $a$ was replaced with $ab$. 
\end{itemize}
These rules are consistent as long as the $b$'s come in pairs (except for a possible lone $b$ at the 
beginning or end 
of a word), separated by an even number of $a$'s. 
This is true
for $F_0$ and $F_1$, and follows inductively from the fact that $ba^{2k}b$ becomes $a(abba)^ka$.

We claim that all of the words $F_n$ are recognisable inflation words of level 1 
with radius 0, with $F_{n-1}$ being the
unique root of $F_n$. By induction, $F_n$ is then a recognisable 
inflation word of level $n$, again with radius 0, with 
unique root $aa$. 
The first few iterates and their inflation word decompositions and corresponding roots are given in the following table.
$$\begin{array}{|c|c|c|}
\hline
n  &    F_n:= L_n \, | \, R_n  & D_{\vartheta^n}(F_n) \\ \hline
1 & ab \, | \, ba & \{([ab,ba],aa)\} \\
2 & baa \, | \, aab & \{([baa,aab],aa)\} \\
3 & aabba \, | \,abbaa & \{([aabba,abbaa],aa)\}\\
4 & abbaaaab \, | \, baaaabba & \{([abbaaaab,baaaabba],aa)\} \\
5 & baaaabbaabbaa \, | \, aabbaabbaaaab & \{([baaaabbaabbaa,aabbaabbaaaab],aa)\}\\
\hline
\end{array}$$

We will show by hand that $F_1$, $F_2$ and $F_3$ have unique 1-decompositions and then generalise to $F_n$. 

Since $b$'s come only from substituting $a$'s, the only way that $bb$ can ever appear in a $\vartheta$-legal word is as
part of $abba$, decomposed as $([ab,ba],aa)$. That
is, $F_1$ only decomposes as $\phi(F_0)$. Note that 
this also shows that $bbb$ is not $\vartheta$-legal. 

$F_2$ contains $aaaa$. Since the middle two $a$'s are not
adjacent to $b$'s, they must come from substituting $b$'s. 
Since $bbb$ is not $\vartheta$-legal, the outer two letters
of the root must be $a$'s, so our root is $abba=F_1$.
This argument also shows that $aaaaa$ is not $\vartheta$-legal, since the middle three $a$'s would have
to all come from $b$'s and $bbb$ is not $\vartheta$-legal. 

$F_3$ contains $abbaabba$. As previously noted, each $abba$ can only come from substituting two $a$'s. Since $aaaaa$ is 
not $\vartheta$-legal, the outer letters of the root must be $b$'s, so the root is $baaaab = F_2$.

When $n>3$, we simply repeat these arguments on each 
portion $baab$ or $baaaab$ of $F_n$. Since $baab$ is preceded and followed by $b$'s, it can only decompose as 
$([ba,ab], aa)$, and since $baaaab$ contains two central $a$'s that can only come from $b$'s, this piece can 
only decompose as $([ba,a,a,ab], abba)$. The
decomposition of any remaining prefixes and suffixes is determined by
the fact that the root cannot contain three successive $b$'s or five successive $a$'s. Thus the only decomposition of $F_n$ is as $\phi(F_{n-1})$.
As $\vartheta$ is irreducible Pisot and admits recognisable words at all levels, it follows from Theorem \ref{pisot-recogwords-nonmixing} that the random Fibonacci subshift is not topologically mixing.

Incidentally, we can also build recognisable bi-infinite words as limits of the 
$F_n$'s. These words form a (deterministic) substitution subshift in their own right, a morphic 
copy of the Fibonacci subshift, with $a$ replaced by $A=baaaab$ and $b$ replaced by $B=baab$. 
This is because the map $\phi$ sends $baaaab$ to $aabbaabbaa$ and sends $baab$ to $aabbaa$.
After applying a conjugation that removes $baa$ from the end of each substituted word and places
it at the beginning, this is equivalent to $A \to AB$, $B \to A$. The fact that such words 
are recognisable under $\vartheta$, and not merely under $\phi$, follows from the same 
arguments as with the finite words $F_n$. Specifically, each $abba$ can only come from $aa$ 
and all of the remaining $a$'s can only come from $b$'s. 
\end{example}

\section{random substitution tiling spaces}\label{SEC:R-actions}

Subshifts are one way to build a dynamical system from a random substitution. An equally valid way
is to build a space of tilings on which $\mathbb{R}$ acts by translation \cite{Baake,S:book}. 
To each letter $a_i \in \mathcal{A}$ we associate an interval of length $\ell(a_i)$, which we 
sometimes abbreviate as $\ell_i$. We call such an interval a {\em tile}. 
To each word $u=u_1\cdots u_n$ we associate an interval of length $\sum_{i=1}^n \ell(u_i)$
obtained by laying tiles of type $u_1, u_2, \ldots, u_n$ end to end. We call such a concatenation
of tiles a {\em patch}. A patch associated to an $n$-super-word is called an $n$-{\em supertile}.
A patch is called $\vartheta$-legal if it is contained in a supertile of some order. A {\em tiling}
is a bi-infinite patch. The {\em tiling space} $\Omega_\vartheta$ associated to the random
substitution $\vartheta$ is the set of tilings with the property that every finite patch is 
$\vartheta$-legal. 

The group $\mathbb{R}$ acts on $\Omega_\vartheta$ by translation. Specifically, if $\mathcal{T}$ is 
a tiling and $t \in \mathbb{R}$, we define $T^t(\mathcal{T})$ to be the tiling obtained by moving
all tiles in $\mathcal{T}$ a distance $t$ to the left. (From the point of view of the tiles, 
this is equivalent to moving the origin a distance $t$.) Instead of explicitly invoking the
group action $T$, we usually just call the result $\mathcal{T}-t$. 

Let $P$ be a finite
patch, understood to lie at a specific location, and let $U$ be an open subset of $\mathbb{R}$. 
Then 
$$\mathcal{Z}_{P,U} = \left\{ \mathcal{S} \in \Omega_\vartheta \mid P \subset \mathcal{S}-y 
\hbox{ for some } y\in U \right\}$$
is called a {\em cylinder set}. By shifting the open set $U$, we can assume, without loss of 
generality, that the left endpoint of $P$ is at the origin. These cylinder sets form a basis
for the topology of $\Omega_\vartheta$. This topology is also the metric topology
coming from the {\em big box metric}, in which two tilings are considered $\epsilon$-close if, 
after translation by up to $\epsilon$, they agree on the interval $[-\epsilon^{-1}, \epsilon^{-1}]$.

The definition of topological mixing for $\mathbb{R}$ actions is similar to that for subshifts.
\begin{definition} The dynamical system $(\Omega_\vartheta, T)$ is topologically mixing if, for
any two open sets $\mathcal{Z}_1$ and $\mathcal{Z}_2$, there is a number $R$ such that, for all
$t>R$, $T^t(\mathcal{Z}_1)\cap \mathcal{Z}_2 \neq \emptyset$. 
\end{definition}
Although this definition is phrased only in terms of large positive values of $t$, topological
mixing is also a property of $T^t$ for $t$ large and negative, since switching the roles
of $\mathcal{Z}_1$ and $\mathcal{Z}_2$ is equivalent to changing the sign of $t$. 

For subshifts, topological mixing is related to the number of letters in possible
return words $w$ between $\vartheta$-legal words $u$ and $v$. That is, the possible lengths of words
$w$ such that $uwv$ is $\vartheta$-legal. Topological mixing for tiling spaces boils down to 
studying the geometric lengths of the {\em patches} associated to those return words. 
Specifically, let $S_{u,v}$ be the set of lengths of all patches associated to words $uw$, where
$uwv$ is $\vartheta$-legal. 

\begin{definition} A discrete subset $S \subset \mathbb{R}$ is said to be {\em asymptotically
dense} if, for each $\epsilon>0$ there exists a number $R_\epsilon$ such that every $t>R_\epsilon$
is within $\epsilon$ of an element of $S$.
\end{definition}

\begin{proposition} The tiling space $\Omega_\vartheta$ is topologically mixing if and only if, 
for all $\vartheta$-legal words $u$ and $v$, $S_{u,v}$ is asymptotically dense.
\end{proposition}

\begin{proof}
Without loss of generality, we can check topological mixing on sets of the form
$\mathcal{Z}_{P,(-\epsilon,\epsilon)}$, since the definition is invariant under translation of
the sets and since intervals of size $2\epsilon$ form a basis for the topology of $\mathbb{R}$. 
We therefore consider sets $\mathcal{Z}_1 = \mathcal{Z}_{P_1, (-\epsilon_1,\epsilon_1)}$
and $\mathcal{Z}_2 = \mathcal{Z}_{P_2, (-\epsilon_2,\epsilon_2)}$, where $P_1$ is a 
patch starting at the origin 
based on the $\vartheta$-legal word $u$ and $P_2$ is a patch starting at the origin based on the
$\vartheta$-legal word $v$. The open set $\mathcal{Z}_1-t$ consists of all tilings that exhibit
a copy of $P_1$ starting within $\epsilon_1$ of $-t$, while $\mathcal{Z}_2$ consists of 
all tilings that exhibit a copy of $P_2$ starting within $\epsilon_2$ of the origin. These sets
intersect if and only if there is a tiling in which the left endpoints of 
$P_1$ and $P_2$ are spaced between
$t-(\epsilon_1+\epsilon_2)$ and $t+(\epsilon_1+\epsilon_2)$ apart. In other words, if and only if
$t$ is within $\epsilon_1+\epsilon_2$ of an element of $S_{u,v}$. The condition that $T^t(\mathcal(Z)_1)\cap \mathcal(Z)_2 \neq \emptyset$ for all sufficiently large $t$ for all 
choices of $\epsilon_1$ and $\epsilon_2$ is equivalent to $S_{u,v}$ being asymptotically dense. 
\end{proof}

Our results for subshifts, and those of \cite{Kenyon}, carry over to tiling spaces with one key adjustment. The condition that 
$\gcd\left\{|\theta^n(a)|:a\in \mathcal{A}\right\}=1$ for all $n\ge 1$ is replaced with the 
existence of letters $a$ and $b$ such that the ratio $\ell(a)/\ell(b)$ is irrational. 

\begin{theorem}[\!\!{\cite[Prop.\ 1.1]{Kenyon}}]\label{thm:KSS-R-action}
Let $\theta$ be a primitive aperiodic deterministic substitution on an alphabet $\mathcal{A}$.
If the tiling space $\Omega_{\theta}$ is topologically mixing, then there exist letters 
$a, b \in \mathcal{A}$ such that $l(a)/l(b)$ is irrational.
\end{theorem}

\begin{proof}
If all the tile lengths are rational multiples of one another, then there is a basic length $\ell_0$
such that all tile lengths are multiples of $\ell_0$. But then every element of $S_{u,v}$ 
(where $u$ and $v$ are arbitrary $\vartheta$-legal words) is a 
multiple of $\ell_0$, so $S_{u,v}$ is not asymptotically dense. 
\end{proof}

Kenyon, Sadun and Solomyak \cite{Kenyon} determined when a tiling space with a 
2-letter deterministic substitution, with second eigenvalue either larger or smaller than 1,
is topologically mixing. (The borderline case where $|\lambda_2|=1$ and $\ell(a)/\ell(b)$ 
is irrational is subtle. There 
are examples where the tiling space is topologically mixing and examples where it isn't.) 

\begin{theorem}[\!\!{\cite[Thm.\ 1.2]{Kenyon}}]\label{thm:KSS-iff-R-action}
Let $\theta$ be a primitive aperiodic deterministic substitution defined on a two-letter alphabet $\mathcal{A}=\{a,b\}$. If $|\lambda_2|<1$ or if the ratio $\ell(a)/\ell(b)$ of the two tile lengths
is rational, then the tiling space $\Omega_\theta$ is not  topologically mixing. If $|\lambda_2|>1$ and the ratio is irrational, then the tiling space is topologically mixing.
\end{theorem}

We now generalise these results to the random setting. Many of the arguments are the same as 
for subshifts and as such will be abbreviated. 

\begin{theorem}
Let $\Omega$ be a tiling space whose allowed sequences of tiles are given by a 
$C$-balanced subshift $X$. Then $\Omega$ is not topologically mixing.
\end{theorem}
\begin{proof}
As in the proof of Theorem \ref{thm:No-mixing-balanced}, there are at most $(C+1)^{d-1}$ abelianisations
of words with a fixed number of letters, where $d$ is the size of our alphabet. 
If $\ell_1$ is the length of the shortest tile, then any patch
of length $L$ has at most $L/\ell_1$ letters. Thus there are at most $L(C+1)^{d-1}/\ell_1$ possible 
patch lengths of size $L$ or less. In particular, the number of elements of $S_{u,v}$ of size $L$ or 
less grows at most linearly with $L$, so $S_{u,v}$ cannot be asymptotically dense. 
\end{proof}

Since random Pisot substitutions are $C$-balanced, we have an immediate corollary that generalises
half of Theorem \ref{thm:KSS-iff-R-action}.
We highlight that this result is independent of any recognisability assumptions on the random substitution.
\begin{theorem}\label{pisot-R-action-nonmixing}
Let $\vartheta$ be a primitive compatible random substitution on an alphabet $\mathcal{A}=\{a_1,\ldots, a_d\}$ with second eigenvalue $\lambda_2$ less than $1$ in modulus.
Then the associated RS-tiling space $\Omega_{\vartheta}$ is not topologically mixing.
\end{theorem}

Finally, we generalise the remainder of Theorem \ref{thm:KSS-iff-R-action}:
\begin{theorem}\label{mainthm-non-pisot-R-action}
Let $\vartheta$ be a primitive compatible random substitution on a two-letter alphabet 
with second eigenvalue $\lambda_2$ greater than $1$ in modulus, such that the ratio of 
the two tile lengths is irrational. Then $\Omega_\vartheta$ is topologically mixing. 
\end{theorem}

\begin{proof}
Let $u$ and $v$ be $\vartheta$-legal words and let $\theta$ be a marginal of $\vartheta$ such
that $u$ and $v$ are $\theta$-legal. The second eigenvalue of $M_\theta$ is a positive power of 
$\lambda_2$ and so is larger than 1 in modulus. Applying Theorem \ref{thm:KSS-iff-R-action} to
$\Omega_\theta$, we see that $\Omega_\theta$ is topologically mixing, which implies that the 
set of lengths
of the return patches from $u$ to $v$ in $\Omega_\theta$ is asymptotically dense. But that set is a 
subset of $S_{u,v}$, so $S_{u,v}$ is asymptotically dense. Since this is true for all pairs $(u,v)$,
$\Omega_\vartheta$ is topologically mixing. 
\end{proof}

\section*{Acknowledgements}
It is a pleasure to thank M.~Baake, D.~Frettl\"{o}h, F.~G\"{a}hler and N.~Ma\~{n}ibo for helpful discussions.
EDM would like to acknowledge the support of the Alexander von Humboldt Foundation and Ateneo de Manila University.
DR was supported by the German Research Foundation (DFG), within the CRC 1283 at Bielefeld University, where the majority of the research was conducted.
GT would like to thank the Commission on Higher Education (CHED) of the Philippines for support.

\bibliography{tilings}
\bibliographystyle{jis}

\end{document}